\documentclass[a4paper]{article}
\pdfoutput=1

\usepackage{fullpage}
\usepackage[T1]{fontenc}
\usepackage{graphicx}
\usepackage{amsmath,amssymb,amsthm}
\usepackage{mathtools}
\usepackage{microtype}
\usepackage[colorlinks = true, linkcolor = black, citecolor = black]{hyperref}
\usepackage{cleveref}
\usepackage{algorithm}
\usepackage{algorithmicx}
\usepackage{algpseudocode}
\usepackage{subcaption}
\usepackage[group-separator={,}]{siunitx}
\algnewcommand\algorithmicinput{\textbf{Input:}}
\algnewcommand\Input{\item[\algorithmicinput]}
\algnewcommand\algorithmicoutput{\textbf{Output:}}
\algnewcommand\Output{\item[\algorithmicoutput]}

\newcommand{\R}{\mathbb{R}}
\newcommand{\bA}{\mathbf{A}}

\DeclareMathOperator{\GL}{GL}

\DeclareMathOperator{\err}{err}

\DeclarePairedDelimiter{\norm}{\lVert}{\rVert}

\newtheorem{theorem}{Theorem}[section]
\newtheorem{lemma}[theorem]{Lemma}
\newtheorem{proposition}[theorem]{Proposition}
\theoremstyle{definition}

\newtheorem{remark}[theorem]{Remark}
\newtheorem*{remark*}{Remark}

\numberwithin{equation}{section}

\title{Numerically stable variants of overrelaxation for operator Sinkhorn iteration}
\author{Henrik Eisenmann\thanks{Institut f\"ur Geometrie und Praktische Mathematik, RWTH Aachen University, 52062 Aachen, Germany} \and Tasuku Soma\thanks{The Institute of Statistical Mathematics, Tokyo, 190-8562, Japan} \and Xun Tang\thanks{Department of Mathematics, Stanford University, Stanford, CA,
94305, USA} \and Andr\'e Uschmajew\thanks{Institute of Mathematics \& Centre for Advanced Analytics and Predictive Sciences, University of Augsburg, 86159 Augsburg, Germany}}
\date{}

\begin{document}

\maketitle
\begin{abstract}
We consider accelerated versions of the operator Sinkhorn iteration (OSI) for solving scaling problems for completely positive maps. Based on the interpretation of OSI as alternating fixed point iteration, it has been recently proposed to achieve acceleration by means of nonlinear successive overrelaxation (SOR), e.g.~with respect to geodesics in Hilbert metric. The direct implementation of the proposed SOR algorithms, however, can be numerically unstable for ill-conditioned instances, limiting the achievable accuracy. Here we derive equivalent versions of OSI with SOR where, similar to the original OSI formulation, scalings are applied on the fly in order to take advantage of preconditioning effects. Numerical experiments confirm that this modification allows for numerically stable SOR-acceleration of OSI even in ill-conditioned cases.
\end{abstract}

\section{Introduction}

Given matrices $A_1, \dots, A_k \in \R^{m \times n}$, the \emph{operator scaling} problem asks for finding invertible matrices $L \in \GL_m(\R)$ and $R \in \GL_n(\R)$ such that the \emph{scaled} matrices
\[
\bar A_i = L A_i R^\top
\]
satisfy the two conditions
\begin{equation}\label{eq: OS}
\begin{aligned}
    \sum_{i=1}^k \bar A_i^{} \bar A_i^\top  &= \sum_{i=1}^k  L A_i^{} R^\top R  A_i^\top L^\top =  \frac{1}{m}I_m ,\\
    \sum_{i=1}^k \bar A_i^\top \bar A_i^{} &= \sum_{i=1}^k R A_i^\top L^\top L A_i^{} R^\top = \frac{1}{n}I_n.
\end{aligned}
\end{equation}
This problem has several theoretical and practical applications in different areas, including in classical and quantum complexity~\cite{Gurvits2004, Garg2019, Allen-Zhu2018}, functional analysis~\cite{Garg2018,Sra2018}, covariance and scatter estimation~\cite{Dutilleul1999,Franks2020,Drton2021}, and signal processing~\cite{Barthe1998,Kwok2018}. Notably, operator scaling can be regarded as a quite far-reaching generalization of the famous matrix scaling problem. The operator scaling problem can in principle also be posed in a complex setting, in which case one asks the sums of the Hermitian matrices $\bar A_i^{} \bar A_i^*$ and $\bar A_i^* \bar A_i^{}$ to be multiples of identity. However, in this work we restrict to the real case.

A particular fruitful interpretation of the operator scaling problem, in particular with regard to the theoretical analysis, is its interpretation as a fixed point problem on cones of positive definite matrices. Let $\mathcal C_m = \{ X \in \mathbb R^{m \times m} \colon X^T = X, X \succ 0\}$ denote the cone of symmetric positive definite matrices in~$\mathbb R^{m \times m}$. Since every $X \in \mathcal C_m$ and $Y \in \mathcal C_n$ can be written as
\[
X = L^\top L, \qquad Y = R^\top R
\]
with $L \in \GL_m(\R)$ and $R \in \GL_n(\R)$, the equations in~\eqref{eq: OS} are equivalent to
\begin{equation}\label{eq: fp formulation}
\begin{aligned}
 X &= \frac{1}{m} \left( \sum_{i=1}^k A_i^{} Y A_i^\top \right)^{-1}, \\
 Y &= \frac{1}{n} \left( \sum_{i=1}^k A_i^\top X A_i^{} \right)^{-1}
 \end{aligned}
\end{equation}
for some $X \in \mathcal C_m$ and $Y \in \mathcal  C_n$. This is a nonlinear fixed point equation on the product cone $\mathcal C_m \times \mathcal C_n$, and the operator scaling problems asks for finding a fixed point. The linear map
\[
\Phi: \mathbb R^{n \times n} \to \mathbb R^{m \times m}, \qquad \Phi(Y) \coloneqq \sum_{i=1}^k A_i^{} Y A_i^\top,
\]
appearing in the first equation is called the \emph{completely positive (CP) map} associated with the matrices $A_1,\dots,A_k$, because it maps the cone of symmetric positive semidefinite matrices in $\mathbb R^{n \times n}$ into the cone of symmetric positive semidefinite matrices in $\mathbb R^{m \times m}$. The map $\Phi^*(X) = \sum_{i=1}^k A_i^\top X A_i^{}$ in the second equation is its adjoint.

\subsection*{The operator Sinkhorn iteration}

The fixed point formulation~\eqref{eq: fp formulation} suggests the following alternating fixed point iteration for finding a solution to the problem: starting from $Y_0 \in \mathcal C_n$, usually $Y_0 = I_n$, compute recursively the sequences
\begin{equation}\label{eq: OSI as fpi}
\begin{aligned}
 X_{t+1} &= \frac{1}{m} \left(\sum_{i=1}^k A_i^{} Y_t^{} A_i^\top \right)^{-1}, \\
 Y_{t+1} &= \frac{1}{n} \left(\sum_{i=1}^k A_i^\top X_{t+1}^{} A_i^{} \right)^{-1}.
\end{aligned}
\end{equation}
This algorithm is called the \emph{operator Sinkhorn iteration (OSI)} due to its analogy to the Sinkhorn algorithm for matrix scaling~\cite{Sinkhorn1967}, and was considered (in an equivalent form discussed below) in seminal work by Gurvits~\cite{Gurvits2004}. In a different context, the same algorithm was actually proposed as \emph{MLE Algorithm} in an earlier work by Dutilleul~\cite{Dutilleul1999} and is also known as the \emph{flip-flop} algorithm~\cite{Brown2001}.

It can be shown that the OSI is well defined (i.e.~the inverses always exist) if the CP maps satisfy $\Phi(I_n) \succ 0$ and $\Phi^*(I_m) \succ 0$. Moreover, using fundamental results from nonlinear Perron--Frobenius theory, it can be shown~\cite{Georgiou2015,Idel2016,SU24} that a sufficient condition for OSI to converge for any initialization to a (unique up to scalar scaling indeterminacy) fixed point $(X_*,Y_*) \in \mathcal C_m \times \mathcal C_n$ of~\eqref{eq: fp formulation} is that the CP map~$\Phi$ is \emph{positivity improving}, which means that it actually maps positive \emph{semidefinite} matrices to positive \emph{definite} ones. Moreover, in this case the asymptotic rate of convergence is linear. In fact, in the original work of Gurvits~\cite{Gurvits2004}, convergence of the OSI has been proven under even weaker conditions (which are in particular satisfied if $\Phi$ is positivity improving), but only with an algebraic $1/\sqrt{t}$ rate.

From the practical perspective, the computation of the inverses in~\eqref{eq: OSI as fpi} is usually based on Cholesky-type decompositions such as
\[
 \sum_{i=1}^k A_i^{} Y_t^{} A_i^\top = C_t^{} C_t^\top
\]
and then setting
\[
 L_{t+1} = \frac{1}{\sqrt{m}} C_t^{-1}, \qquad X_{t+1} = L_{t+1}^\top L_{t+1}^{}.
\]
Likewise, in the next half step
\[
 \sum_{i=1}^k A_i^\top X_{t+1}^{} A_i^{} = D_t^{} D_t^\top
\]
and
\[
 R_{t+1} = \frac{1}{\sqrt{n}} D_t^{-1}, \qquad Y_{t+1} = R_{t+1}^\top R_{t+1}^{}.
\]
In this way, a sequence of scaling matrices $(L_t,R_t)$ is generated while executing the algorithm, and if OSI converges in the sense that $(X_t,Y_t) = (L_t^\top L_t^{}, R_t^\top R_t^{} )$ converges to a fixed point $(X_*,Y_*) \in \mathcal C_m \times \mathcal C_n$ of~\eqref{eq: fp formulation}, then every accumulation point $(L_*,R_*)$ of the sequence $(L_t,R_t)$ provides a solution for the operator scaling problem~\eqref{eq: OS}. For reference, the algorithm is displayed as \Cref{algo: FPI}. Note that the corresponding generated sequences $(X_t,Y_t)$ do not depend on the particular choice of the Cholesky-type decompositions, so there is a certain freedom of choice. The most reasonable choices are to take as $C_t$ and $D_t$ either the classic lower-triangular Cholesky factors, which can be computed and inverted quite efficiently, or to take $C_t$ and $D_t$ as positive definite square roots, which has advantages for the theoretical analysis and will deliver symmetric positive definite scaling matrices $L_*$ and $R_*$. Moreover, with such a systematic choice there will be usually only one accumulation point $(L_*,R_*)$ of the sequence $(L_t,R_t)$.

\begin{figure*}[t]
\begin{minipage}[t]{.49\linewidth}
\begin{algorithm}[H]
\caption{\small Alternating Fixed-point Iteration}
\label{algo: FPI}
\begin{algorithmic}[1]
%\Require
\small
\Input
Matrices $A_1,\dots,A_k \in \mathbb R^{m \times n}$, invertible $R_0$.
%\Output
% \State Initialize scaling factors $L_0 = I_m$, $R_0 = I_n$.
\For
{$t = 0,1,2,\dots$}
    \State
    Update $L$ factor:
    \[
    \sum_{i=1}^k A_{i}^{} R_t^\top R_t^{} A_{i}^\top = C_t^{} C_t^\top, \quad     L_{t+1} = \frac{1}{\sqrt{m}} C_t^{-1}.
    \]
    \State
    Update $R$ factor:
    \[
    \sum_{i=1}^k A_{i}^{\top} L_{t+1}^\top L_{t+1}^{} A_{i}^{} = D_t^{} D_t^\top,
    \quad
    R_{t+1} = \frac{1}{\sqrt{n}} D_t^{-1}.
    \]
\EndFor
\end{algorithmic}
\end{algorithm}
\end{minipage}
\hfill
\begin{minipage}[t]{.49\linewidth}
\begin{algorithm}[H]
\caption{\small Original OSI}
\label{algo: OSI}
\begin{algorithmic}[1]
\small
%\Require
\Input
Initial matrices $\bar A_{1,0} \coloneqq A_1,\dots, \bar A_{k,0} \coloneqq A_k$.
%\Output
\For
{$t = 0,1,2,\dots$}
    \State
    Update $L$ factor:
    \[
    \sum_{i=1}^k \bar A_{t,i}^{} \bar A_{t,i}^\top = \bar C_t^{} \bar C_t^\top,
    \quad
    \bar L_{t+1} = \frac{1}{\sqrt{m}} \bar C_t^{-1}.
    \]
    \State
    Update $R$ factor:
    \[
    \sum_{i=1}^k \bar A_{t,i}^{\top} \bar L_{t+1}^\top \bar L_{t+1}^{} \bar A_{t,i}^{} = \bar D_t \bar D_t^{\top},
    \quad
    \bar R_{t+1} = \frac{1}{\sqrt{n}} \bar D_t^{-1}.
    \]
    \State
    Absorb factors by setting
    \[
    \bar A_{t+1,i} = \bar L_{t+1}^{} \bar A_{t,i} \bar R_{t+1}^\top.
    \]
\EndFor
\end{algorithmic}
\end{algorithm}
\end{minipage}
\end{figure*}

We now discuss an important equivalent version of OSI related to the main subject of our work. As originally presented in~\cite{Gurvits2004}, the OSI can be formulated in a way where the scaling matrices $L_t$ and $R_t$ are immediately absorbed into the matrices $A_1,\dots,A_k$ by replacing them with $L_t^{} A_i^{} R_t^\top$, and then continuing the procedure with these new matrices. This way of formulating the iteration somewhat puts more emphasis on the resulting scaled $A_i$ instead of the scaling matrices $L$ and $R$. We display this version as \Cref{algo: OSI}. Note that this algorithm implicitly takes $R_0 = I_m$ as initial guess, essentially without loss in generality. It can be shown by an easy induction~\cite{SU24} that if also $R_0 = I_m$ in \Cref{algo: FPI}, then both variants are equivalent up to orthogonal transformation in the sense that the sequences of scaled matrices $\bar A_{t,i}$ produced by \Cref{algo: OSI} are related to the sequences $L_t$ and $R_t$ from \Cref{algo: FPI} through
\[
 L_t^{} A_i^{} R_t^\top = P_t^{} \bar A_{t,i}^{} Q_t^\top
\]
for some orthogonal matrices $P_t \in \mathbb R^{m \times m}$ and $Q_t \in \mathbb R^{n \times n}$ (which arise since the choice of Cholesky decomposition has not been specified). In particular, if $(X_t,Y_t)$ in \Cref{algo: FPI} converges to a fixed point $(X_*,Y_*) \in \mathcal C_m \times \mathcal C_n$, then for the same data the $\bar A_{t,i}$ in \Cref{algo: OSI} converge to matrices $\bar A_i$ satisfying~\eqref{eq: OS}. If the scaling matrices themselves are required, they can also be obtained from \Cref{algo: OSI} through recursive book-keeping by noting that
\begin{equation}\label{eq: recursive scaling}
\bar A_{t,i} = \bar L_t^{} \cdots \bar L_1^{} A_i^{} \bar R_1^\top \cdots \bar R_t^\top.
\end{equation}

The two version of OSI are mathematically essentially equivalent, but can behave differently numerically though. In~\Cref{algo: FPI}, the matrices $\sum_{i=1}^k A_{i}^{} R_t^\top R_t^{} A_{i}^\top$ and $\sum_{i=1}^k A_{i}^{\top} L_{t+1}^\top L_{t+1}^{} A_{i}^{}$ of which Cholesky decompositions need to be computed are supposed to converge to $m X_*^{-1}$ and $n Y_*^{-1}$, respectively, which may become difficult if one of these limits is ill-conditioned. In fact, already forming these matrices can be numerically unstable, if the initial matrices $A_1,\dots,A_k$ are ill-conditioned. The achievable accuracy of~\eqref{eq: fp formulation} can hence be affected by these factors. Examples of this effect can be seen in the numerical experiments in \Cref{sec: experiment Hilbert matrices}. In contrast, in~\Cref{algo: OSI} often a regularizing effect is observed, because in this algorithm one has to compute Cholesky-type decompositions of $\sum_{i=1}^k \bar A_{t,i}^{} \bar A_{t,i}^\top$ and $\sum_{i=1}^k \bar A_{t,i}^{\top} \bar L_{t+1}^\top \bar L_{t+1}^{} \bar A_{t,i}^{}$ which both should converge to multiples of identities and hence become better and better conditioned. Therefore, the original version of OSI as in \Cref{algo: OSI} can be considered as numerically more stable, which is also confirmed in the experiments. The goal of the present work is to develop accelerated versions of \Cref{algo: OSI} based on overrelaxation, which can significantly increase the convergence while keeping the same numerical stability.

\subsection*{Successive overrelaxation for operator scaling}

In~\cite{SU24} several methods for accelerating OSI by means of (nonlinear) successive overrelaxation (SOR) have been proposed, building on earlier works on SOR methods for classic matrix scaling~\cite{Thibault2021,Lehmann2022}. The SOR methods considered in~\cite{SU24} are based on the alternating fixed-point formulation as given in \Cref{algo: FPI}. Recall that this version of OSI is based on the formulation~\eqref{eq: OSI as fpi} of OSI as alternating fixed point iteration. Following a classical pattern for such alternating methods, the main idea of SOR is to combine current and previous iterates in the algorithm using a relaxation parameter $\omega > 0$ for improving the convergence rate. In the case of the iteration~\eqref{eq: OSI as fpi} a complication arises from the fact that the iterates $X_t$ and $Y_t$ must be symmetric positive definite matrices, which renders a naive linear relaxation such as taking $(1-\omega) X_t + \omega X_{t+1}$ feasible only for values $0 \le \omega \le 1$. However, the case of interest in overrelaxation is $\omega > 1$.

To adapt to the geometry of positive definite matrices, it has been proposed in~\cite{SU24} to instead either apply overrelaxation in the linear space of lower triangular Cholesky factors, or using geodesics in the hyperbolic space of positive definite matrices equipped with the Hilbert metric. The first variant is straightforward and displayed as \Cref{algo: FPI-Cholesky-SOR}. This algorithm comes at almost no additional cost compared to \Cref{algo: FPI} (restricted to lower-triangular Cholesky decomposition). When $\omega > 1$, there is a theoretical risk that $L_{t+1}$ or $R_{t+1}$ become non-invertible, but in practice this should rarely pose problems. The geodesic version of SOR is a little bit more sophisticated, as it requires the geodesic curve in the Hilbert metric connecting positive definite matrices $X$ and $\tilde X$, which is given by~\cite{Lemmens2012}
\begin{align}\label{eq:sharp}
X \#_\omega \tilde X \coloneqq X^{\frac{1}{2}}(X^{-\frac{1}{2}}\tilde X X^{-\frac{1}{2}})^\omega X^{\frac{1}{2}}, \qquad 0 \le \omega \le 1,
\end{align}
but actually well defined for all $\omega \in \mathbb R$. Based on this observation, the geodesic SOR version of OSI can be formulated as \Cref{algo: FPI-Geodesic-SOR} which operates directly on the matrices $X_t = L_t^\top L_t^{}$ and $Y_t = R_t^\top R_t^{}$.

\begin{figure*}[t]
\begin{minipage}[t]{.49\linewidth}
\begin{algorithm}[H]
\small
\caption{\small Alternating Fixed-Point Iteration with Cholesky Factor SOR (FPI-Cholesky-SOR)}
\label{algo: FPI-Cholesky-SOR}
\begin{algorithmic}[1]
%\Require
\Input
Matrices $A_1,\dots,A_k \in \mathbb R^{m \times n}$, invertible matrices $L_0, R_0$, and relaxation parameter $\omega >0$.
%\Output
\For
{$t = 0,1,2,\dots$}
    \State
    Update $L$ factor:
    \begin{align*}
        &\sum_{i=1}^k A_{i}^{} R_t^\top R_t^{} A_{i}^\top = C_t^{} C_t^\top, 
    \\
        &L_{t+1} = (1-\omega) L_t + \frac{\omega}{\sqrt{m}} C_t^{-1}.
    \end{align*}
    with $C_t$ lower triangular.
    \State
    Update $R$ factor:
    \begin{align*}
    &\sum_{i=1}^k A_{i}^{\top} L_{t+1}^\top L_{t+1}^{} A_{i}^{} = D_t^{} D_t^\top, \\
    &R_{t+1} = (1-\omega ) R_t + \frac{\omega}{\sqrt{n}} D_t^{-1}.
    \end{align*}
    with $D_t$ lower triangular.
\EndFor
\end{algorithmic}
\end{algorithm}
\end{minipage} 
\hfill
%%%%%%%%%%%%%%%%%%%%%%%%%%%%%%%%%%%%%%%%%%%%%%%%%%%%
\begin{minipage}[t]{.49\linewidth}
\begin{algorithm}[H]
\small
\caption{\small OSI with Cholesky Factor SOR\\(OSI-Cholesky-SOR)}
\label{algo: OSI-Cholesky-SOR}
\begin{algorithmic}[1]
%\Require
\Input
Initial matrices $A_{1,0} \coloneqq A_1,\dots,A_{k,0} \coloneqq A_k$.
%\Output
\For
{$t = 0,1,2,\dots$}
    \State
    Update $L$ factor:
    \begin{align*}
    &\sum_{i=1}^k \bar A_{t,i}^{} \bar A_{t,i}^\top = \bar C_t^{} \bar C_t^\top,  \\
    &\bar L_{t+1} = (1-\omega) I_m + \frac{\omega}{\sqrt{m}} \bar C_t^{-1}.
    \end{align*}
    with $C_t$ lower triangular.
    \State
    Update $R$ factor:
    \begin{align*}
    &\sum_{i=1}^k \bar A_{t,i}^{\top} \bar L_{t+1}^\top \bar L_{t+1}^{} \bar A_{t,i}^{} = \bar D_t \bar D_t^{\top}, \\
    &\bar R_{t+1} = (1-\omega) I_n + \frac{\omega}{\sqrt{n}} \bar D_t^{-1}.
    \end{align*}
    with $D_t$ lower triangular.
    \State
    Absorb factors by setting
    \[
    \bar A_{t+1,i} = \bar L_{t+1}^{} \bar A_{t,i} \bar R_{t+1}^\top.
    \]
\EndFor
\end{algorithmic}
\end{algorithm}
\end{minipage}
\end{figure*}

\begin{figure*}[t]
\begin{minipage}[t]{.49\linewidth}
\begin{algorithm}[H]
\small
\caption{Alternating Fixed-point Iteration with Geodesic SOR (FPI-Geodesic-SOR)}
\label{algo: FPI-Geodesic-SOR}
\begin{algorithmic}[1]
%\Require
\Input
Matrices $A_1,\dots,A_k \in \mathbb R^{m \times n}$, initial matrices $X_0, Y_0$, and relaxation parameter $\omega >0$.
%\Output
\For
{$t = 0,1,2,\dots$}
    \State
    Update $X$:
    \begin{align*}
    X_{t+1} = X_t \#_\omega \left[\frac{1}{m} \left(\sum_{i=1}^k A_i Y_t A_i^\top \right)^{-1} \right].
    \end{align*}
    \State
    Update $Y$:
    \begin{align*}
     Y_{t+1} = Y_t \#_\omega \left[ \frac{1}{n} \left(\sum_{i=1}^k A_i^\top X_{t+1} A_i \right)^{-1} \right].
    \end{align*}
\EndFor
\end{algorithmic}
\end{algorithm}
\end{minipage}
\hfill
%%%%%%%%%%%%%%%%%%%%%%%%%%%%%%%%%%%%%%%%%%%%    
\begin{minipage}[t]{.49\linewidth}
\begin{algorithm}[H]
\small
\caption{OSI with geodesic SOR\\(OSI-Geodesic-SOR)}
\label{algo: OSI-Geodesic-SOR}
\begin{algorithmic}[1]
%\Require
\Input
Initial matrices $A_{1,0} \coloneqq A_1,\dots,A_{k,0} \coloneqq A_k$.
%\Output
\For
{$t = 0,1,2,\dots$}
    \State
    Update $L$ factor:
    \[
    \bar L_{t+1} = m^{-\omega/2} \left( \sum_{i=1}^k \bar A_{t,i}^{} \bar A_{t,i}^\top\right)^{-\omega/2}.
    \]
    \State
    Update $R$ factor:
    \[
    \bar R_{t+1} = n^{-\omega/2} \left(\sum_{i=1}^k \bar A_{t,i}^{\top} \bar L_{t+1}^\top \bar L_{t+1}^{} \bar A_{t,i}^{}\right)^{-\omega/2}.
    \]
    \State
    Absorb factors by setting
    \[
    \bar A_{t+1,i} = \bar L_{t+1}^{} \bar A_{t,i} \bar R_{t+1}^\top.
    \]
    % {\color{red} State version with book keeping?}
\EndFor
\end{algorithmic}
\end{algorithm}
\end{minipage}
\end{figure*}

Both \Cref{algo: FPI-Cholesky-SOR} and \Cref{algo: FPI-Geodesic-SOR} apply the idea of overrelaxation to the alternating fixed point iteration~\eqref{eq: OSI as fpi}. This has the advantage that some standard arguments from the convergence analysis of such algorithms can be applied, at least in modified form. Note that both SOR variants preserve the fixed points~\eqref{eq: fp formulation}, and as shown in~\cite{SU24} their linearizations actually coincide in these fixed points and take the form of a linear SOR iteration. In other words, in a neighborhood of a solution both algorithms are equal in first-order. Furthermore, it is then possible to characterize the ``optimal'' parameter $\omega$ for achieving the fastest \emph{asymptotic} local convergence rate. Unfortunately, to the best of our knowledge, the \emph{global} convergence properties of OSI with SOR remain currently unclear.

However, as also observed in~\cite{SU24}, both \Cref{algo: FPI-Cholesky-SOR} and~\ref{algo: FPI-Geodesic-SOR} can suffer from ill-conditioning and limited numerical accuracy in the same way as the standard version in \Cref{algo: FPI}: for instance when the matrices $A_1,\dots,A_k$ are ill-conditioned. This is insofar expectable, as they both also do not benefit from a regularizing effect of an on-the-fly scaling as in \Cref{algo: OSI}. In this short note, we present equivalent versions of the SOR algorithms from~\cite{SU24} which apply the scaling matrices already during the iteration and are hence natural SOR generalizations of the original OSI as in~\Cref{algo: OSI}. These algorithms are displayed as \Cref{algo: OSI-Cholesky-SOR} and \Cref{algo: OSI-Geodesic-SOR} and will be derived in the next section. As we will demonstrate in the numerical experiments in \Cref{sec: numerical experiments}, the proposed SOR versions can significantly accelerate the original SOR and solve the operator scaling problem to high accuracy even in ill-conditioned scenarios.

\section{SOR variants of OSI}\label{sec: SOR-OSI}

Starting with the initial matrices $A_i \eqcolon \bar A_{0,i}$ ($i=1,\dots,k$), the original OSI (Algorithm~\ref{algo: OSI}) produces a recursive sequence
\[
\bar A_{t+1,i} = \bar L_{t+1}^{} \bar A_{t,i}^{} \bar R_{t+1}^\top, \qquad i=1,\dots,k,
\]
of successively scaled matrices which are supposed to converge to a solution $\bar A_i = L_* A_i R_*$ ($i=1,\dots,k$) of the operator scaling problem~\eqref{eq: OS}. The formulation of SOR variants of \Cref{algo: OSI} is based on the simple observation that after applying the scaling, one obtains new matrices $A_{t+1,i}$ and should implicitly assume in the next iteration that the ``previous'' scaling matrices are just identity matrices. Apart from that the same SOR update rules as in the fixed point versions (\Cref{algo: FPI-Cholesky-SOR} and \Cref{algo: FPI-Geodesic-SOR}) can be used for suitably combining the standard updates with identities. An alternative viewpoint for the same idea is that every iteration of \Cref{algo: OSI} looks like an iteration of \Cref{algo: FPI} with $L_t = I_m$ and $R_t = I_n$, and with different matrices $A_{i,t}$ instead of $A_i$. Hence SOR update formulas can be modified accordingly. In the following, we derive the two corresponding algorithms for the additive Cholesky factor SOR update rule as used in \Cref{algo: FPI-Cholesky-SOR} and the geodesic SOR rule as in \Cref{algo: FPI-Geodesic-SOR}, respectively, and prove that they are indeed equivalent to their alternating fixed point version in the same way as \Cref{algo: OSI} is equivalent to \Cref{algo: FPI}.

\subsection{OSI with Cholesky factor SOR}

In this variant, we assume that in the original OSI always standard Cholesky decompositions without pivoting (i.e.~with lower-triangular factors) are used. We then apply linear SOR directly to the computed inverse Cholesky factors, implicitly assuming the ``previous'' Cholesky factors have been identity matrices. The resulting algorithm is displayed as \Cref{algo: OSI-Cholesky-SOR}, and should be compared with \Cref{algo: FPI-Cholesky-SOR}. We next show that these algorithms are mathematically equivalent. To show this, we first make the following observation on Cholesky decompositions which follows immediately from the fact that the Cholesky (upper-triangular) decomposition of (real) symmetric positive definite matrices is unique.

\begin{lemma}
    Let $A$ be positive definite, $L_1$ be lower triangular so that $L_1^{} L_1^\top = A$ and $L_2$ are lower triangular with positive entries. Then $(L_2L_1)^{}(L_2L_1)^\top = L_2^{} A L_2^\top$ is the Cholesky decomposition of $L_2^{} A L_2^\top$.    
\end{lemma}

We now proceed to the main statement.

\begin{proposition}\label{prop: equivalence Cholesky}
    Let $(L_t,R_t)$ denote the iterates of \Cref{algo: FPI-Cholesky-SOR} when initialized with $L_0 = I_m$ and $R_0 = I_n$. Accordingly, let $(\bar L_t, \bar R_t, \bar A_{i,t})$ be the iterates of \Cref{algo: OSI-Cholesky-SOR} for the same data $A_1,\dots,A_k$ and same $\omega > 0$. Then for all $t \ge 0$ it holds $L_{t+1} = \bar L_{t+1} L_t $, $R_{t+1} = \bar R_{t+1} R_t $,
    and
    \[
    \bar A_{t,i} = L_t^{} A_i^{} R_t^\top
    \]
    for $i=1,\dots,k$.
\end{proposition}

\begin{proof}
    The proof is by induction. In the case $t=0$ there is nothing to show. Assume now the statements are correct for some $t \ge 0$. Since then
    \[
    \bar C_t^{} \bar C_t^\top
    = \sum_{i=1}^k  \bar A_{t,i}^{}\bar A_{t,i}^\top
    = \sum_{i=1}^k L_t^{} A_{i}^{} R_t^\top R_t^{} A_{i}^\top L_t^\top 
    =  L_t^{}C_t^{} C_t^\top L_t^\top,
    \]
    the above lemma implies $\bar C_t = L_tC_t$. Therefore,
    \[
    L_{t+1} = (1-\omega) L_t +\frac\omega{\sqrt m} C_t^{-1}= (1-\omega) L_t +\frac\omega{\sqrt m }  (L_t C_t)^{-1} L_t
    = ((1-\omega) I_m +\frac\omega{\sqrt m }  \bar C_t^{-1}) L_t = \bar L_{t+1} L_t.
    \]
    Similarly, by
    \[
        \bar D_t^{} \bar D_t^{\top}
        =
        \sum_{i=1}^k \bar A_{t,i}^{\top} \bar L_{t+1}^\top \bar L_{t+1}^{} \bar A_{t,i}^{} 
        =
    \sum_{i=1}^k R_t^{} A_{i}^{\top} L_{t+1}^\top L_{t+1}^{} A_{i}^{} R_t^\top = R_t D_t^{} D_t^\top R_t^\top,
    \]
    we have $\bar D_t = R_t^{} D_t^{}$, and thus
    \[
    R_{t+1}
    =
    (1-\omega)R_t +\frac\omega{\sqrt{n}} D_t^{-1}
    =
    ((1-\omega)I_n +\frac\omega{\sqrt{n}} \bar D_t^{-1}) R_t = \bar R_{t+1} R_t.
    \]
    Finally, we have
    \[
        \bar A_{t+1,i} = \bar L_{t+1}^{} \bar A_{t,i}^{} \bar R_{t+1}^\top 
        = \bar L_{t+1}^{} L_{t}^{} A_i^{}  R_{t}^\top R_{t+1}^\top = L_{t+1}^{} A_{t,i}^{} R_{t+1}^\top,
    \]
    completing the induction step.
\end{proof}

\begin{remark}\label{remark1}
    The use of an ``additive'' SOR update rule for combining old and new iterates in Algorithms~\ref{algo: FPI-Cholesky-SOR} and~\ref{algo: FPI-Geodesic-SOR} is motivated by the fact that the triangular Cholesky decomposition provides a (local) parametrization of symmetric positive definite matrices through a linear space of lower triangular matrices. In principle, one could ignore this underlying property and use instead the matrix square root for computing $C_t,D_t$ and $\bar C_t, \bar D_t$ in both algorithms, respectively. We experimented with such a version of \Cref{algo: OSI-Cholesky-SOR} using $\bar C_t = (\sum_{i=1}^k \bar A_{t,i}^{} \bar A_{t,i}^\top )^{1/2}$ and similar for $\bar D_t$. It showed comparable performance and numerical stability. An interesting observation, however, is that (according to the experiments) this version does not seem to be equivalent to a corresponding square root version of \Cref{algo: FPI-Cholesky-SOR}. Apart from that also note that one cannot choose \emph{any} Cholesky-type decomposition in \Cref{algo: OSI-Cholesky-SOR}. To preserve the fixed point property of the solution, the choice needs to be such that the ``Cholesky factor'' of the identity matrix should be the identity matrix itself. This basically leaves the classic lower triangular Cholesky factor and the matrix square root as only reasonable choices.
\end{remark}

\subsection{OSI with geodesic SOR}

Using the same idea as before, we could also absorb the scaling factors in the geodesic version of \Cref{algo: FPI-Geodesic-SOR}. This then just means that in the next iterations, we find ourselves in the situation from the initial iterate, just with modified matrices $A_i$. This suggests that we can combine the geodesic SOR with OSI by connecting the standard update with the identity matrices. For example, for updating the $L$ factor in the $t$-th step this would mean to take
\begin{equation}\label{eq: derivation Geodesic OSI}
\bar X_{t+1} = I_m \#_\omega \left [\frac{1}{m} \left(\sum_{i=1}^k \bar A_{t,i}^{} \bar A_{t,i}^\top \right)^{-1} \right] = \left( \frac1m \sum_{i=1}^k \bar A_{t,i}^{} \bar A_{t,i}^\top \right)^{-\omega} 
\end{equation}
and then set $\bar L_{t+1}$ to be a Cholesky factor of this matrix. If we now agree to take the matrix square root as Cholesky factor, then we can immediately write
\[
\bar L_{t+1} = \left( \frac1m \sum_{i=1}^k \bar A_{t,i}^{} \bar A_{t,i}^\top \right)^{-\omega/2}
\]
to obtain the new left scaling matrix. In a similar way we can obtain a direct formula for $\bar R_{t+1}$. The resulting iteration is denoted as \Cref{algo: OSI-Geodesic-SOR} and should be compared with \Cref{algo: FPI-Geodesic-SOR}. An obvious difference between both is that \Cref{algo: FPI-Geodesic-SOR} operates on the variables $X = L^\top L$ and $Y = R^\top R$, while \Cref{algo: OSI-Geodesic-SOR} operates on $L$ and $R$ factors directly, which are needed here because immediate scaling is applied. In exact arithmetic, both algorithms are equivalent in the following sense.

\begin{proposition}

Let $(X_t,Y_t)$ denote the iterates of \Cref{algo: FPI-Geodesic-SOR} when initialized with $(X_0, Y_0) = (I_m, I_n)$. Accordingly, let $(\bar L_t, \bar R_t, \bar A_{t,i})$ be the iterates of \Cref{algo: OSI-Geodesic-SOR} for the same data $A_1,\dots,A_k$ and same $\omega > 0$. Define
\[
L_{t+1} = \bar L_{t+1} L_t, \qquad R_{t+1} = \bar R_{t+1} R_t
\]
recursively with $L_0 = I_m$ and $R_0 = I_n$. Then for all $t \ge 0$ it holds 
\[
\bar A_{t,i} = L_t A_i R_t
\]
for $i=1,\dots,k$ and 
\[
X_t = L_t^\top L_t^{} , \qquad Y_t = R_t^\top R_t^{}.
\]
\end{proposition}

\begin{proof}
The equation for $\bar A_{t,i}$ holds by construction, see~\eqref{eq: recursive scaling}. For $t= 0$, the statements for $X_t$ and $Y_t$ are trivial since all involved matrices are the identity matrices. We proceed by induction and assume the statements are true for some $t \ge 0$. Note that 
\[
 \frac{1}{m} \left(\sum_{i=1}^k A_i^{} Y_t^{} A_i^\top \right)^{-1}
 =
 \frac{1}{m}
 \left(\sum_{i=1}^k  L_t^{-1}\bar A_{t,i}^{}  \bar A_{t,i}^\top L_t^{-\top}\right)^{-1}
 =
 L_t^{\top}\bar L_{t+1}^{2/\omega} L_t^{}.
\]
Thus, \Cref{algo: FPI-Geodesic-SOR} constructs
\[
    X_{t+1} = X_t \#_\omega (L_t^{\top}\bar L_{t+1}^{2/\omega} L_t^{}) 
    = X_t^{\frac{1}{2}} (X_t^{-\frac{1}{2}} L_t^{\top}\bar L_{t+1}^{2/\omega} L_t^{} X_t^{-\frac{1}{2}})^\omega X_t^{\frac{1}{2}}.
\]
Now since $L_t^\top L_t^{} = X_t$ by induction hypothesis, we have $X_t^{-\frac12}L_t^\top L_t^{} X_t^{-\frac12} = I_m$, so $L_t X^{-\frac12}$ is an orthogonal matrix. Since analytic functions on symmetric matrices satisfy $f(Q^\top A Q) =Q^\top f(A) Q$ for orthogonal $Q$, the above expression for $X_{t+1}$ hence simplifies to
\[
X_{t+1} = X_t^{\frac{1}{2}} X_t^{-\frac{1}{2}} L_t^{\top} (\bar L_{t+1}^{2/\omega})^\omega L_t^{} X_t^{-\frac{1}{2}} X_t^{\frac{1}{2}}
= L^\top_t \bar L_{t+1}^2 L_t^{} = L_{t+1}^\top L_{t+1}^{},
\]
since $\bar L_{t+1}$ is symmetric. Similarly, using that $ R_t^{} Y_t^{-\frac{1}{2}}$ is orthogonal, one now also verifies
\begin{align*}
    Y_{t+1} &= Y_t ^{\frac{1}{2}} \left[ Y_t ^{-\frac{1}{2}} 
    \frac{1}{n} \left(\sum_{i=1}^k A_i^\top X_{t+1}^{} A_i^{} \right)^{-1}
    Y_t ^{-\frac{1}{2}} \right]^{\omega}Y_t ^{\frac{1}{2}} \\
&=
Y_t ^{\frac{1}{2}} \left[
Y_t ^{-\frac{1}{2}} R_t^\top \frac{1}{n}
\left(\sum_{i=1}^k 
\bar A_{t,i}^\top \bar L_{t+1}^\top \bar L_{t+1}^{} \bar A_{t,i}^{}  \right)^{-1}
    R_t^{} Y_t^{-\frac{1}{2}}\right]^{\omega}
    Y_t^{\frac{1}{2}}\\
&=
 R_t^\top n^{-\omega} \left(\sum_{i=1}^k 
\bar A_{t,i}^\top \bar L_{t+1}^\top
    \bar L_{t+1}^{} \bar A_{t,i}^{} 
    \right)^{-\omega}
    R_t^{} = R_t^\top \bar R_{t+1}^2 R_t^\top = R_{t+1}^\top R_{t+1}^{},
\end{align*}
concluding the induction step.
\end{proof}

\begin{remark}
Similar to \Cref{remark1} one may ask whether \Cref{algo: OSI-Geodesic-SOR} necessarily requires to use the matrix square root as ``Cholesky factor''. In general, one may think of using $\bar L_{t+1} = m^{-\omega/2} \bar C_t^{-\omega}$, where~$\bar C_t$ is any matrix satisfying $\bar C_t^{} \bar C_t^\top = \sum_{i=1}^k \bar A_{t,i}^{} \bar A_{t,i}^\top$ (and analogously for the $R$ step). We experimented with a version of \Cref{algo: OSI-Geodesic-SOR} using classic lower triangular Cholesky factors, which is insofar interesting as lower-triangular matrices are also closed under taking rational powers. In particular, choosing the triangular Cholesky factor in \Cref{algo: OSI-Geodesic-SOR} ensures the fixed point property of the identity matrix. However, we were not able to relate such a version of \Cref{algo: OSI-Geodesic-SOR} to the \Cref{algo: FPI-Geodesic-SOR}. Regardless of that, we found that such a version of \Cref{algo: OSI-Geodesic-SOR} behaves quite similar in regard of number of iterations and numerical stability, but was somewhat slower in execution time compared to the matrix square root version, at least when using off-the-shelf functions for rational powers of nonsymmetric functions.
\end{remark}

\section{Numerical experiments}\label{sec: numerical experiments}

We present experimental results for the solution of the operator scaling problem~\eqref{eq: OS} in special, ill-conditioned instances. The main goal is to demonstrate that the proposed SOR versions of OSI can achieve much higher numerical accuracy in such cases compared to their their mathematically equivalent FPI counterparts. Specifically, the comparison of the following pairwise equivalent methods is of interest:

\begin{itemize}
    \item FPI (Algorithm~\ref{algo: FPI}) and OSI (Algorithm~\ref{algo: OSI}),
    \item FPI-Cholesky-SOR (Algorithm~\ref{algo: FPI-Cholesky-SOR}) and OSI-Cholesky-SOR (Algorithm~\ref{algo: OSI-Cholesky-SOR}),
    \item FPI-Geodesic-SOR (Algorithm~\ref{algo: FPI-Geodesic-SOR}) and OSI-Geodesic-SOR (Algorithm~\ref{algo: OSI-Geodesic-SOR}).
\end{itemize}

In addition to that, we will also observe differences between the different OSI methods, e.g.~in runtime. However, while we tried to achieve comparably efficient implementations of all algorithms, this was not our focus, and so the results on runtime may only indicate a relative trend.

\subsection{Experimental Setup}

All experiments were run on a Linux server with AMD EPYC 9654 CPU with 192 logical cores and \num{755}GB memory using Python 3.10.14, NumPy~1.26.4, and SciPy~1.10.1. The initialization of all algorithms were set to the identity matrices, i.e. $L_0 = I_m$ and $R_0 = I_n$ for the FPI methods and $\bar A_{0,i} = A_i$ for the OSI methods. We use the following error measure called ``grad norm'':
\[    
    \err(\bA_t) = \sqrt{\norm*{\sum_{i=1}^k A_{t,i}^{} A_{t, i}^\top - \frac{1}{m} I_m}_F^2 + \norm*{\sum_{i=1}^k A_{t,i}^\top A_{t, i}^{} - \frac{1}{n} I_n}_F^2}.
\]
We denote the error at iteration $t$ by $\err_t$. The overrelaxation in all SOR methods was activated after $p=5$ iterations with $\omega$ computed with the following ``optimal'' formula (see~\cite{SU24}):
\begin{align*}
    \omega = \frac{2}{1 + \sqrt{1- \hat\beta^2}} > 1,
\end{align*}
where
\begin{align*}
    {\hat \beta}^2 = \sqrt{\frac{\err_p}{\err_{p-2}}}
\end{align*}
is an estimate of the local convergence rate of the corresponding method without SOR.

\subsection{Ill-conditioned problem with Hilbert matrices}\label{sec: experiment Hilbert matrices}

As a first numerical test, we reproduce the experimental result from~\cite{SU24} in which numerical instabilities of the fixed-point based SOR methods have been observed. We set $A_i = Q_i H \in \R^{n \times n}$, $i=1,\dots,k$, where~$H$ is the $n \times n$ Hilbert matrix and $Q_i$ is a uniform random orthogonal matrix.

In \Cref{fig: Hilbert matrix}, we show the experimental outcome for the case $n=5$ and $k=7$. The condition number of every $A_i$ is then about~\num{476600}. We can clearly see the difference between the FPI (fixed-point based) methods and the OSI methods which apply scaling during the iteration; the former getting stuck at error measure around $10^{-6}$. In addition, we see a (small) accelerating effect in all the SOR methods. 
Note that the SOR methods automatically detected $\omega \approx 1.096$. 
Notably, OSI-Cholesky-SOR and OSI-Geodesic-SOR have about the same computational time in our implementation.
On the other hand, OSI-Geodesic-SOR is much faster than its FPI couterpart (FPI-Geodesic-SOR) because the former operates on better-conditioned matrices and hence saving runtime in the inner geodesic computation.

\begin{figure}[t]
    \centering
    \subcaptionbox{Plot of the gradient norm against iterations.\label{fig:Hilbert_iter}}{\includegraphics[width = .49\textwidth]{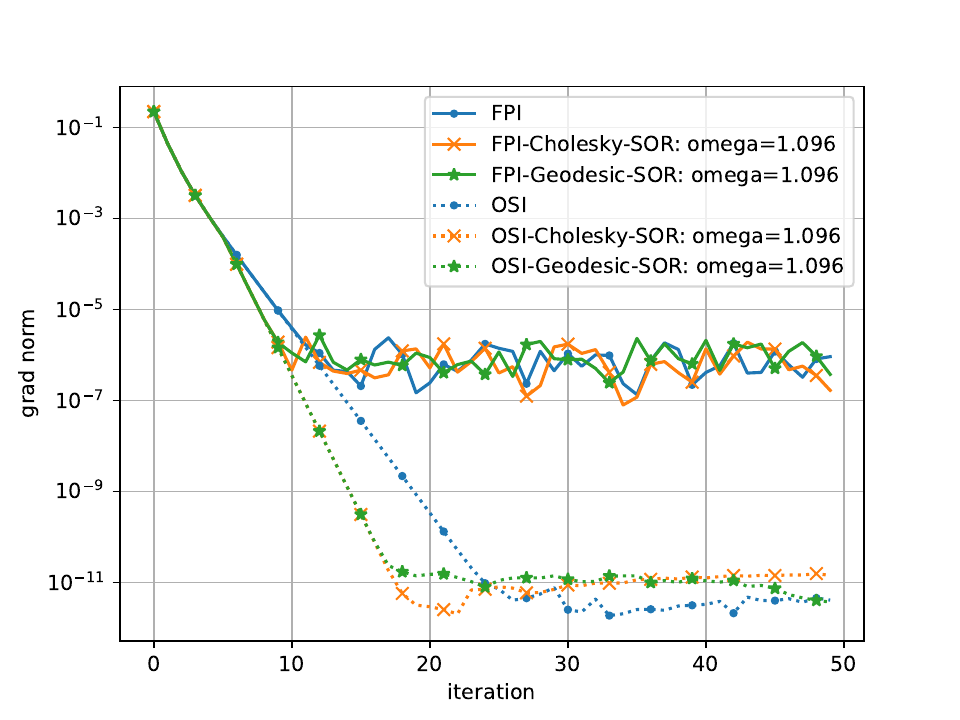}}
    \subcaptionbox{Plot of the gradient norm against runtime.\label{fig:Hilbert_time}}{\includegraphics[width = .49\textwidth]{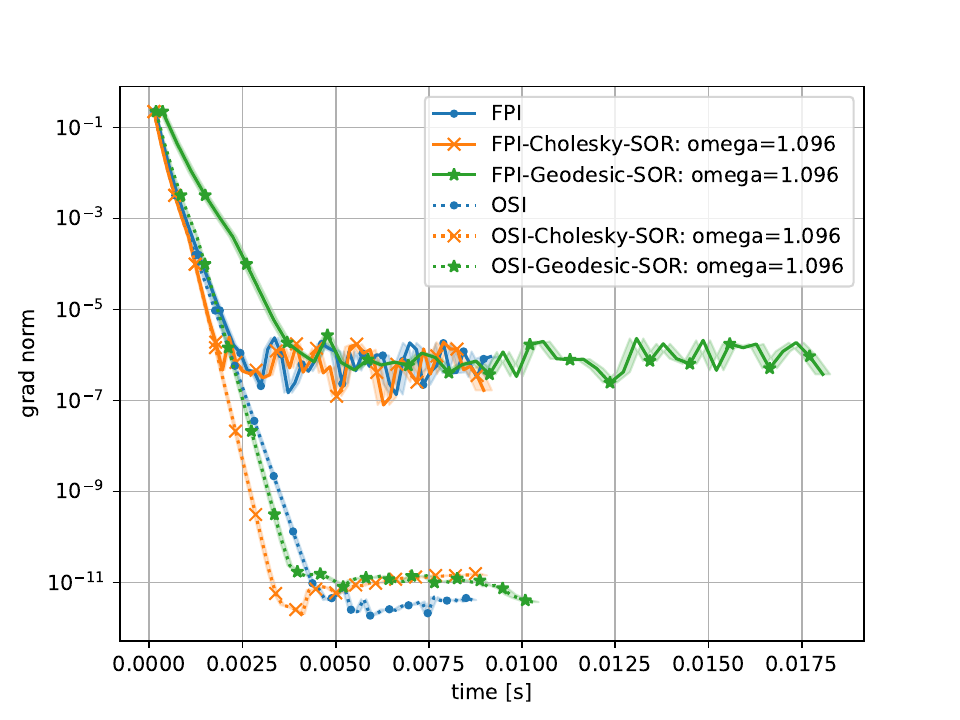}}
\caption{The results of the experiment with Hilbert matrices. The FPI methods are shown in solid lines, whereas the OSI methods are shown in dashed lines. In the runtime plot (\subref{fig:Hilbert_time}), the thick line and shaded region represent the mean and the standard deviation of 200 runs on the same instance. In both plots, the FPI methods suffer from numerical instability and get stuck at around $10^{-6}$, while the OSI methods can achieve much smaller error.}
\label{fig: Hilbert matrix}
\end{figure}

\subsection{An example of a frame scaling instance}

From the above experiment, it can be seen that the proposed OSI-SOR methods can deal with ill-conditioned data matrices $A_i$. However, at the same time the benefit of using overrelaxation is not very large as standard OSI already converges very fast. In the following, we present an ill-conditioned instance of frame scaling in which standard OSI is slow and can be significantly accelerated using SOR. 

Recall that in frame scaling, we are given vectors $x_1, \dots, x_k \in \R^n$, and the goal is to find an invertible $n \times n$ matrix $P$ and scalars $\alpha_i \in \R$ ($i = 1, \dots, k$) such that 
\begin{align*}
    \sum_{i=1}^k \alpha_i^2(Px_i)(Px_i)^\top = I_n,
    \qquad 
    \alpha_i^2 (Px_i)^\top (Px_i) = \frac{n}{k} \quad (i=1, \dots, k).
\end{align*}
It can be seen that frame scaling is a special case of operator scaling where $A_i = x_i^{} e_i^\top$ are rank-one.

In this experiment, we set up a frame scaling instance as follows. Let $Q, P$ be random orthogonal matrices of size $k, n$, respectively. Let $D$ be an $n \times n$ diagonal matrix with its diagonal entries evenly spaced in the interval $[\kappa^{-1}, 1]$, where $\kappa > 1$ is a parameter. Let $B = Q' D P^\top$, where $Q'$ denotes the $k \times n$ submatrix of $Q$ containing its first $n$ columns. Note that the condition number of $B$ is roughly proportional to~$\kappa$. Then, for $i = 1, \dots, k$, let $x_i$ be (the transpose of) the $i$th row of $B$ normalized to Euclidean norm one. In the experiment, we set $n=50$, $k=55$, and $\kappa=10^7$. We run the methods for $200$ iterations and SOR is activated at $p=20$th iteration.

\begin{figure}[t]
\centering
\subcaptionbox{Plot of the gradient norm against iterations.\label{fig:frame_iter}}{\includegraphics[width = .66\textwidth]{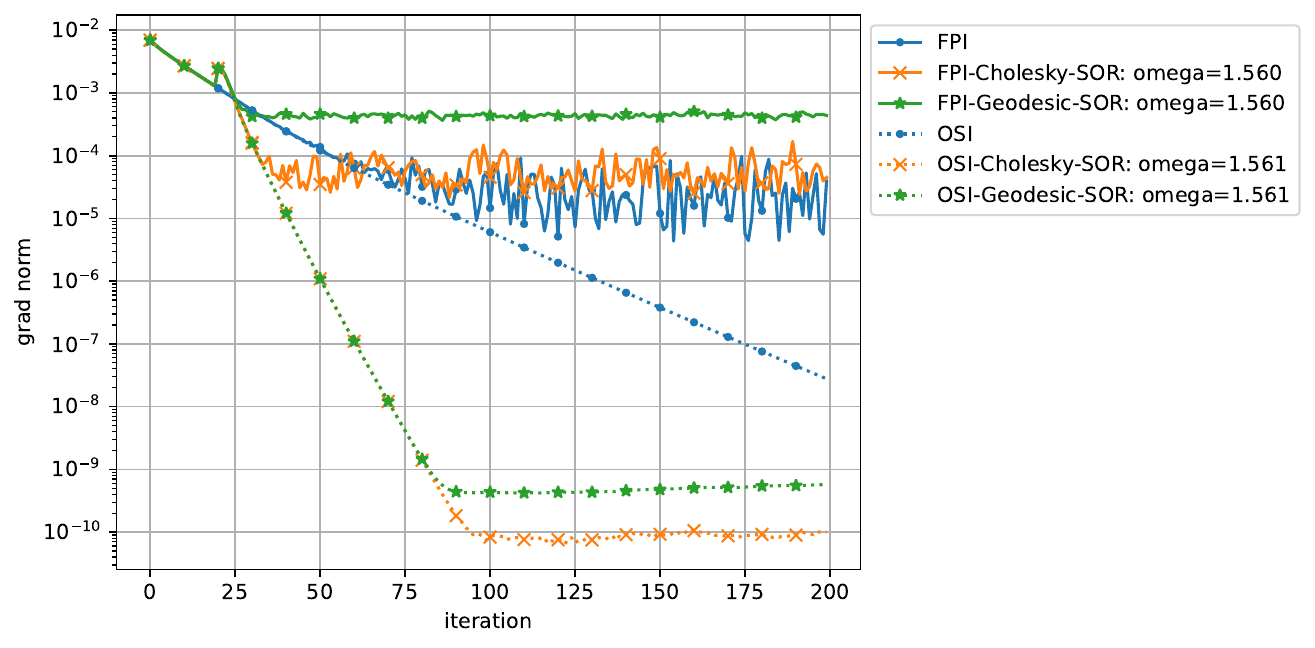}}
\subcaptionbox{Plot of the gradient norm against runtime.\label{fig:frame_time}}{\includegraphics[width = .66\textwidth]{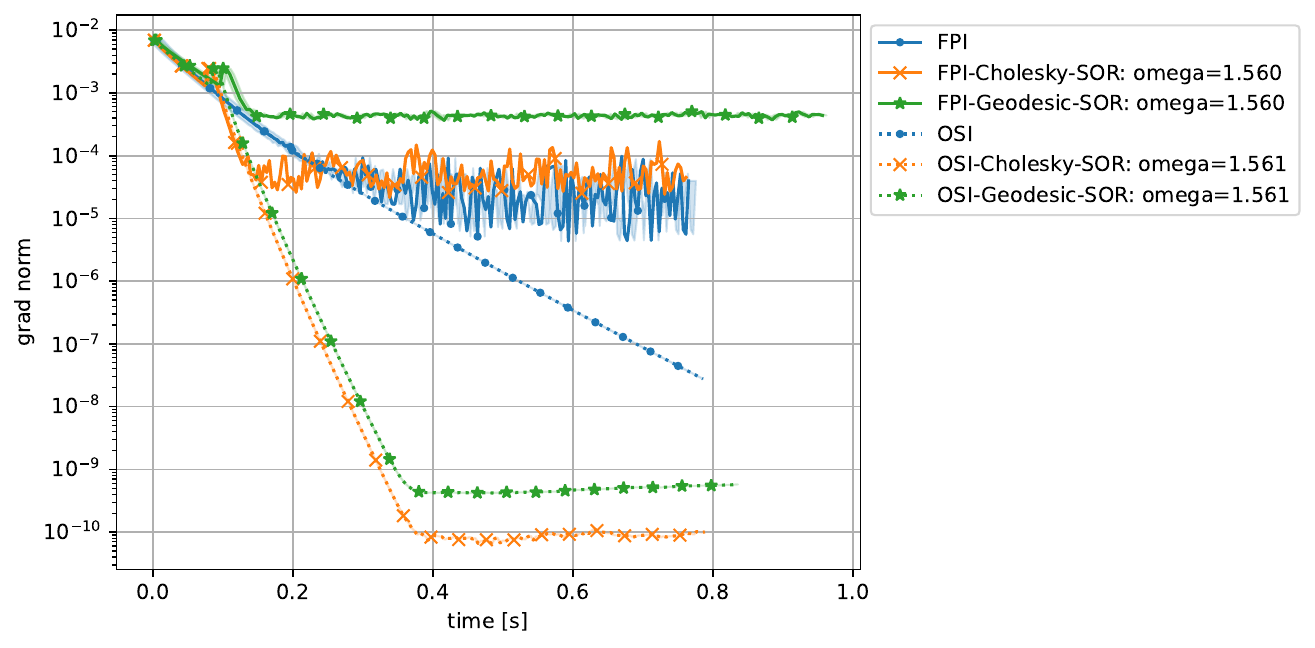}}
\caption{The results of the experiment with frame scaling. The FPI methods are shown in solid lines, whereas the OSI methods are shown in dashed lines. In the runtime plot (\subref{fig:frame_time}), the thick line and shaded region represent the mean and the standard deviation of 10 runs on the same instance, although the shaded region in not visible in the plot because the variance is small. In both plots, the FPI methods suffer from numerical instability and get stuck at around $10^{-3}$ to $10^{-5}$, while the OSI methods can achieve much smaller error. Moreover, SOR significantly accelerate the OSI methods, while only slightly accelerating the FPI methods before getting stuck.}
\label{fig: frame}
\end{figure}

The result is shown in \Cref{fig: frame}. As is expected, the FPI methods stagnated in error around $10^{-3}$ to $10^{-5}$ due to the ill-conditioning, whereas the OSI methods achieved smaller error around $10^{-9}$ to $10^{-10}$. Furthermore, we observe a significant acceleration of the OSI-SOR methods over OSI. The FPI-SOR also accelerated FPI, but got stuck due to numerical error almost immediately after SOR was activated.

\subsection{An extreme case}

We repeat the previous experiment with smaller $k = 52$ (we set $n = 50$, same as in the previous experiment).
Additionally, we replace the first matrix $A_1$ with $e_1^{} e_1^\top$ (a single entry $1$ in the top left corner). In this case, we observe that standard OSI is incredibly slow, while the SOR methods still work (see \Cref{fig: frame_twist}). Note that since the SOR methods converge fast, it is likely that $\omega$ has been almost correctly estimated (the bumps indicate it has been slightly overestimated), which in turn indicates that the ultimate convergence rate of standard OSI is indeed that slow.

It should be noted that in the case of frame scaling, neither map $\Phi(Y) = \sum_{i=1}^k A_i^{} Y A_i^\top$, nor $\Phi^*(X) = \sum_{i=1}^k A_i^\top X A_i^{}$ are positivity improving, as $\Phi(e_i^{} e_i^\top) = x_i^{}x_i^\top$ has rank one and $\Phi^*(y^{} y^\top) = \sum_{i =1}^k (x_i^\top y)^2 e_i^{} e_i^\top$ is not full rank for any $y$ orthogonal to any of the vectors $x_1, \ldots, x_k$. Thus, global convergence rates such as \cite[Theorem~6] {Georgiou2015} are not applicable, and we believe this may be related to the very slow convergence rates of standard OSI observed in this experiment.

\begin{figure}[t]
\centering
\subcaptionbox{Plot of the gradient norm against iterations.\label{fig:frame-twist_iter}}{\includegraphics[width = .66\textwidth]{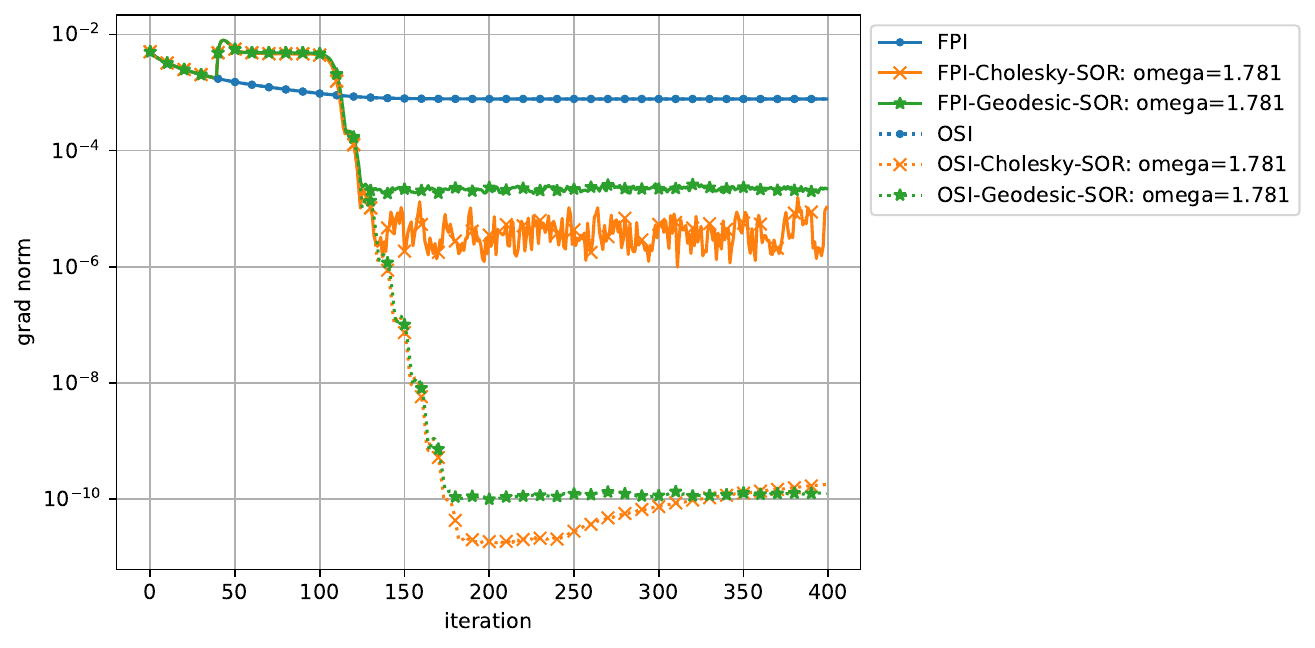}}
\subcaptionbox{Plot of the gradient norm against runtime.\label{fig:frame-twist_time}}{\includegraphics[width = .66\textwidth]{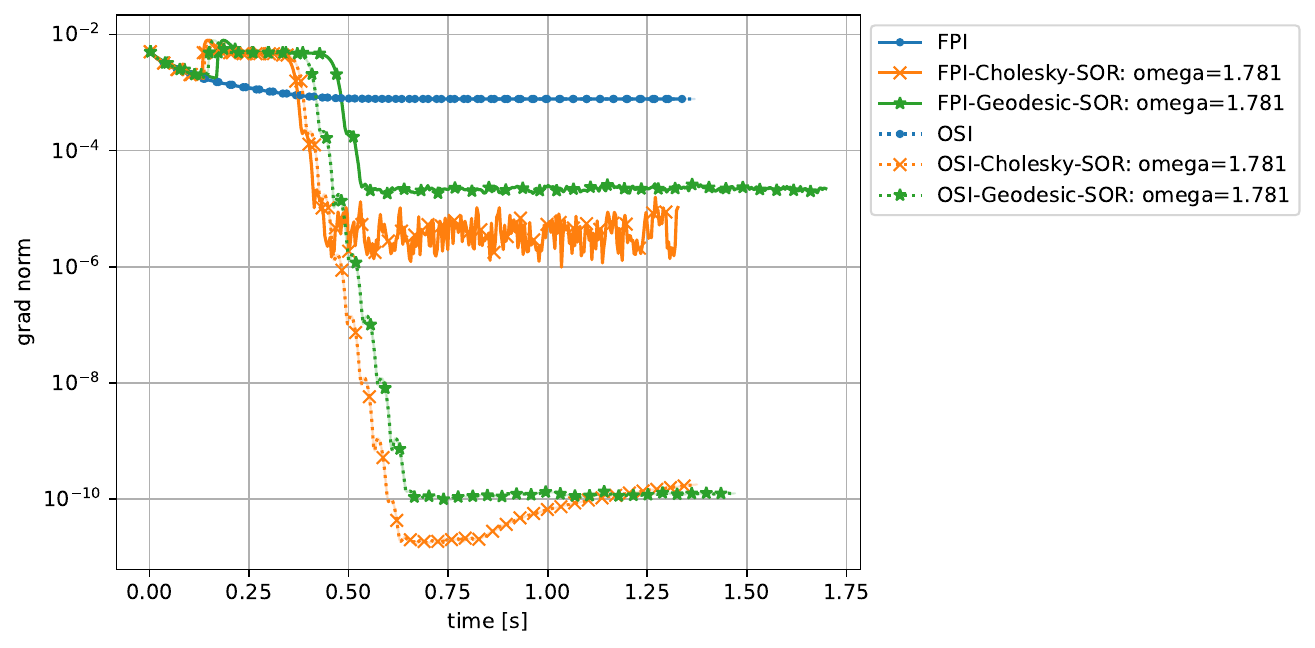}}
\caption{The results of the experiment with the extreme setting. The FPI methods are shown in solid lines, whereas the OSI methods are shown in dashed lines. In the runtime plot (\subref{fig:frame-twist_time}), the thick line and shaded region represent the mean and the standard deviation of 10 runs on the same instance, although the shaded region in not visible in the plot because the variance is small. FPI and OSI (without SOR) are very slow. The FPI methods with SOR are slightly better but get stuck due to numerical instability. On the other hand, the OSI-SOR methods exhibit significant acceleration as well as better convergence error.}
\label{fig: frame_twist}
\end{figure}

\paragraph{Acknowledgements}

The work of A.U.~was supported by the Deutsche Forschungs\-gemeinschaft (DFG, German Research Foundation) – Projektnummer 506561557. A part of this work has been conducted while A.U. was staying at ISM as a visiting professor in 2025. The support of ISM is gratefully acknowledged. T.S.~was supported by JPSP KAKENHI Grant Numbers JP19K20212 and JP24K21315, and JST, PRESTO Grant Number JPMJPR24K5, Japan. The work of H.E.~was supported by Deutsche Forschungs\-gemeinschaft (DFG, German Research Foundation) – Projektnummer 501389786.

\small
\bibliographystyle{alpha}
\bibliography{main}

\end{document}